\documentclass[a4paper,11pt,reqno]{amsart}
\usepackage{amssymb,latexsym}
\usepackage{mathrsfs, epsfig}
\usepackage{graphicx}

\theoremstyle{plain}
\newtheorem{theorem}{Theorem}[section]
\newtheorem{lemma}[theorem]{Lemma}
\newtheorem{corollary}[theorem]{Corollary}
\newtheorem{proposition}[theorem]{Proposition}

\theoremstyle{definition}
\newtheorem{definition}[theorem]{Definition}
\newtheorem{remark}[theorem]{Remark}

\newtheorem{Assumption}[theorem]{Assumption}

\def\grad{\mathop{\rm grad}\nolimits}

\def\nor{\mathop{\rm nor}\nolimits}
\def\inj{\mathop{\rm inj}\nolimits}

\def\argmin{\mathop{\rm argmin}\nolimits}
\def\cut{\mathop{\rm Cut}\nolimits}
\def\arccosh{\mathop{\rm arccosh}\nolimits}
\def\arcsinh{\mathop{\rm arcsinh}\nolimits}
\def\lip{\mathop{\rm Lip}\nolimits}
\makeatletter
\newcommand{\Rmnum}[1]{\expandafter\@slowromancap\romannumeral#1@}

\makeatother

\begin{document}

\title{Some properties of Fr\'{e}chet medians in Riemannian manifolds}



\author{Le Yang}
\address{Laboratoire de Math\'{e}matiques et Applications, CNRS :
UMR 6086, Universit\'{e} de
Poitiers\\
T\'{e}l\'{e}port 2 - BP 30179, Boulevard Marie et Pierre Curie,
86962 Futuroscope Chasseneuil Cedex, France}
\email{Le.Yang@math.univ-poitiers.fr}


\begin{abstract}
The consistency of Fr\'{e}chet medians is proved for probability
measures in proper metric spaces. In the context of Riemannian
manifolds, assuming that the probability measure has more than a
half mass lying in a convex ball and verifies some concentration
conditions, the positions of its Fr\'{e}chet medians are estimated.
It is also shown that, in compact Riemannian manifolds, the
Fr\'{e}chet sample medians of generic data points are always unique.
\end{abstract}


\thanks{The author is supported by a PhD fellowship Allocation de
Recherche MRT and Thales Air Systems.}

\maketitle


\section{Introduction}
The history of medians can be dated back to 1629 when P. Fermat
initiated a challenge (see \cite{Fermat}): given three points in the
plan, find a fourth one such that the sum of its distances to the
three given points is minimum. The answer to this question, which
was firstly found by E. Torricelli (see \cite{Torricelli}) in 1647,
is that if each angle of the triangle is smaller than $2\pi/3$, then
the minimum point is such that the three segments joining it and the
vertices of the triangle form three angles equal to $2\pi/3$; and in
the opposite case, the minimum point is the vertex whose angle is
greater than or equal to $2\pi/3$. This point is called the median
or the Fermat point of the triangle.

The notion of median also appears in statistics since a long time
ago. In 1774, when P. S. Laplace tried to find an appropriate notion
of the middle point for a group of observation values, he introduced
in \cite{Laplace} ``the middle of probability'', the point that
minimizes the sum of its absolute differences to data points, this
is exactly the one dimensional median.

A sufficiently general notion of median in metric spaces is proposed
in 1948 by M. Fr\'{e}chet in his famous article \cite{Frechet},
where he defines a $p$-mean of a random variable $X$ to be a point
which minimizes the expectation of its distance at the power $p$ to
$X$. In practice, two important cases are $p=1$ and $p=2$, which
correspond to the notions of Fr\'{e}chet median and Fr\'{e}chet
mean, respectively. Probably the most significant advantage of the
median over the mean is that the former is robust but the latter is
not, that is to say, the median is much less sensitive to outliers
than the mean. Roughly speaking (see \cite{Lopuhaa}), in order to
move the median of a group of data points to arbitrarily far, at
least a half of data points should be moved. On the contrary, in
order to move the mean of a group of data points to arbitrarily far,
it suffices to move one data point. So that medians are in some
sense more prudent than means, as argued by M. Fr\'{e}chet. The
robustness property makes the median an important estimator in
situations when there are lots of noise and disturbing factors.

Under the framework of Riemannian manifolds, the existence and
uniqueness of local medians are proved in \cite{Yang} for
probability measures whose support are contained in a convex
geodesic ball. It should be noted that, if the local curvature
conditions in \cite{Yang} are replaced by global ones, then it is
shown in \cite{Afsari} that the local medians are in fact global
medians, that is, Fr\'{e}chet medians. Stochastic algorithms and
deterministic algorithms for computing medians can be found in
\cite{Arnaudon4} and \cite{Yang}.

The aim of this paper is to give some basic properties of
Fr\'{e}chet medians. Firstly, we consider the question of
consistency, which is important on the estimation of location.
Theorem \ref{consistency} states that, in proper metric spaces, if
the first moment functions of a sequence of probability measures
converge uniformly to the first moment function of another
probability measure, then the corresponding sequence of Fr\'{e}chet
median sets also converges to the Fr\'{e}chet median sets of the
limiting measure. As a result, if a probability measure has only one
Fr\'{e}chet median, then any sequence of empirical Fr\'{e}chet
medians will converge almost surely to it. In the second section, we
study the robustness of Fr\'{e}chet medians in Riemannian manifolds.
In Euclidean spaces, it is shown in \cite{Lopuhaa} that if a group
of data points has more than a half concentrated in a bounded
region, then its Fr\'{e}chet median cannot be drown arbitrarily far
when the other points move. A generalization and refinement of this
result for data points in Riemannian manifolds is given in Theorem
\ref{position}, where an upper bound of the furthest distance to
which Fr\'{e}chet medians can move is given in terms of the upper
bound of sectional curvatures, the concentration radius and the
concentrated mass of the probability measure. This theorem also
generalizes a result in \cite{Afsari} which states that if the
probability measure is supported in a strongly convex ball, then all
its Fr\'{e}chet medians lie in that ball. Moreover, though we have
chosen the framework of studying the robustness of Fr\'{e}chet
medians to be Riemannian manifolds, it is easily seen that our
results remain true for general CAT$(\Delta)$ spaces. Finally, the
uniqueness question of Fr\'{e}chet sample medians is considered in
the context of compact Riemannian manifolds. It is shown that, apart
from several events of probability zero, the Fr\'{e}chet sample
medians are unique if the sample vector has a density with respect
to the canonical Lebesgue measure of the product manifold. In other
words, the Fr\'{e}chet medians of generic data points are always
unique.

\section{consistency of fr\'{e}chet medians in metric spaces}

Let $(M,d)$ be a proper metric space (recall that a metric space is
proper if and only if every bounded and closed subset is compact)
and $P_1(M)$ denote the set of all the probability measures $\mu$ on
$M$ verifying
 \[\int_M d(x_0,p)\mu(dp)<\infty,\,\,\text{for\,\,some}\,\,x_0\in M.\]
For every $\mu\in P_1(M)$ we can define a function
\[f_{\mu}:\qquad M \longrightarrow\mathbf{R}_+\,,\qquad x\longmapsto\int_M d(x,p)\mu(dp).\]
This function is 1-Lipschitz hence continuous on $M$. Since $M$ is
proper, $f_{\mu}$ attains its minimum (see \cite[p. 42]{Sahib}), 
so we can give the following definition:

\begin{definition}
Let $\mu$ be a probability measure in $P_1(M)$, then a \emph{global}
minimum point of $f_{\mu}$ is called a Fr\'{e}chet median of $\mu$.
The set of all the Fr\'{e}chet medians of $\mu$ is denoted by
$Q_{\mu}$. Let $f^{*}_{\mu}$ denote the global minimum of $f_{\mu}$.
\end{definition}

Observe that $Q_{\mu}$ is compact, since the triangle inequality
implies that $d(x,y)\leq 2f^*_{\mu}$ for every $x,y\in Q_{\mu}$.

To introduce the next proposition, let us recall that the
$L^1$-Wasserstein distance between two elements $\mu$ and $\nu$ in
$P_1(M)$ is defined by
\[W_1(\mu,\nu)=\inf_{\pi\in\Pi(\mu,\,\nu)}\int_{M\times M}d(x,y)d\pi(x,y),\]
where $\Pi(\mu,\nu)$ is the set of all the probability measures on
$M\times M$ with margins $\mu$ and $\nu$. As a useful case for us,
observe that $f_{\mu}(x)=W_1(\delta_x,\mu)$ for every $x\in M$. The
set of all the 1-Lipschitz functions on $M$ is denoted by
$\lip_1(M)$.

As is well known that Riemannian barycenters are characterized by
convex functions (see \cite[Lemma 7.2]{Kendall1}), the following
proposition shows that Fr\'{e}chet medians can be characterized by
Lipschitz functions.

\begin{proposition}\label{lip}
Let $\mu\in P_1(M)$ and $M$ be also separable, then
$$Q_{\mu}=\bigg\{x\in M: \varphi(x)\leq f^*_{\mu}+\int_M \varphi(p)\mu(dp),\,\,\text{for every}\,\,\varphi\in
\lip_1(M)\bigg\}.$$
\end{proposition}

\begin{proof}
The separability of $M$ ensures that the duality formula of
Kantorovich-Rubinstein (see \cite[p. 107]{Villani}) can be applied,
so that for every $x\in M$,
\begin{align*}
x\in Q_{\mu}&\Longleftrightarrow f_{\mu}(x)\leq
f^*_{\mu}\\&\Longleftrightarrow W_1(\delta_x,\mu)\leq f^*_{\mu}\\
&\Longleftrightarrow\sup_{\varphi\in\lip_1(M)}\bigg|\varphi(x)-\int_M \varphi(p)\mu(dp)\bigg|\leq f^*_{\mu}\\
&\Longleftrightarrow\varphi(x)\leq f^*_{\mu}+\int_M
\varphi(p)\mu(dp),\,\,\text{for every}\,\,\varphi\in \lip_1(M),
\end{align*}
as desired.
\end{proof}

We proceed to show the main result of this section.

\begin{theorem}\label{consistency}
Let $(\mu_n)_{n\in\mathbf{N}}$ be a sequence in $P_1(M)$ and $\mu$
be another probability measure in $P_1(M)$. If $(f_{\mu_n})_n$
converges uniformly on $M$ to $f_{\mu}$, then for every
$\varepsilon>0$, there exists $N\in\mathbf{N}$, such that for every
$n\geq N$ we have
\[Q_{\mu_n}\subset B(Q_{\mu},\varepsilon):=\{x\in M: d(x,
Q_{\mu})<\varepsilon\}.\]
\end{theorem}

\begin{proof}
We prove this by contradiction. Suppose that the assertion is not
true, then without loss of generality, we can assume that there
exist some $\varepsilon>0$ and a sequence $(x_n)_{n\in \mathbf{N}}$
such that $x_n\in Q_{\mu_n}$ and $x_n\notin B(Q_{\mu},\varepsilon)$
for every $n$. If this sequence is bounded, then by choosing a
subsequence we can assume that $(x_n)_n$ converges to a point
$x_*\notin B(Q_{\mu},\varepsilon)$ because $M\setminus
B(Q_{\mu},\varepsilon)$ is closed. However, observe that the uniform
convergence of $(f_{\mu_n})_n$ to $f_{\mu}$ implies
$f^*_{\mu_n}\longrightarrow f^*_{\mu}$, hence one gets
\begin{align*}
|f_{\mu}(x_*)-f^*_{\mu}|&\leq
|f_{\mu}(x_*)-f_{\mu_n}(x_*)|+|f_{\mu_n}(x_*)-f_{\mu_n}(x_n)|+|f_{\mu_n}(x_n)-f^*_{\mu}|\\
&\leq \sup_{x\in
M}|f_{\mu_n}(x)-f_{\mu}(x)|+d(x_*,x_n)+|f^*_{\mu_n}-f^*_{\mu}|\longrightarrow
0.
\end{align*}
So that $f_{\mu}(x_*)=f^*_{\mu}$, that is to say $x_*\in Q_{\mu}$.
This is impossible, hence $(x_n)_n$ is not bounded. Now we fix a
point $\bar{x}\in Q_{\mu}$, always by choosing a subsequence we can
assume that $d(x_n,\bar{x})\longrightarrow+\infty$, then
\begin{align}\label{infty}
f_{\mu}(x_n)&=\int_M
d(x_n,p)\mu(dp)\geq\int_M(d(x_n,\bar{x})-d(\bar{x},p))\mu(dp)\notag\\
&=d(x_n,\bar{x})-f^*_{\mu}\longrightarrow+\infty.
\end{align}
On the other hand,
\begin{align*}
|f_{\mu}(x_n)-f^*_{\mu}|&\leq
|f_{\mu}(x_n)-f_{\mu_n}(x_n)|+|f_{\mu_n}(x_n)-f^*_{\mu}|\\
&\leq\sup_{x\in
M}|f_{\mu_n}(x)-f_{\mu}(x)|+|f^*_{\mu_n}-f^*_{\mu}|\longrightarrow
0.
\end{align*}
This contradicts \eqref{infty}, the proof is complete.
\end{proof}

\begin{remark}\label{W1}
A sufficient condition to ensure the uniform convergence of
$(f_{\mu_n})_n$ on $M$ to $f_{\mu}$ is that
$W_1(\mu_n,\mu)\longrightarrow0$, since
\[\sup_{x\in
M}|f_{\mu_n}(x)-f_{\mu}(x)|=\sup_{x\in
M}|W_1(\delta_x,\mu_n)-W_1(\delta_x,\mu)|\leq W_1(\mu_n,\mu).\]
\end{remark}

The consistency of Fr\'{e}chet means is proved in \cite[Theorem
2.3]{Bhattacharya}. The consistency of Fr\'{e}chet medians given
below is a corollary to Theorem \ref{consistency}. A similar result
can be found in \cite[p. 44]{Sahib}.

\begin{corollary}
Let $(X_n)_{n\in\mathbf{N}}$ be a sequence of i.i.d random variables
of law $\mu\in P_1(M)$ and $(m_n)_{n\in\mathbf{N}}$ be a sequence of
random variables such that $m_n\in Q_{\mu_n}$ with
$\mu_n=\frac{1}{n}\sum_{k=1}^{n}\delta_{X_k}$. If $\mu$ has a unique
Fr\'{e}chet median $m$, then $m_n\longrightarrow m$ a.s.
\end{corollary}

\begin{proof}
By Theorem \ref{consistency} and Remark \ref{W1}, it suffices to
show that $\mu_n\xrightarrow{W_1}\mu$ a.s. This is equivalent to
show that (see \cite[p. 108]{Villani}) for every $f\in C_b(M)$,
\[\frac{1}{n}\sum_{k=1}^{n}f(X_k)\longrightarrow\int_M f(p)\mu(dp)\quad \text{a.s.}\]
and for every $x\in M$,
\[\frac{1}{n}\sum_{k=1}^{n}d(x,X_k)\longrightarrow\int_M
d(x,p)\mu(dp)\quad \text{a.s.}\] These two assertions are trivial
corollaries to the strong law of large numbers, hence the result
holds.
\end{proof}

\section{robustness of fr\'{e}chet medians in riemannian manifolds}
Throughout this section, we assume that $M$ is a complete Riemannian
manifold with dimension no less than 2, whose Riemannian distance is
denoted by $d$. We fix a closed geodesic ball
\[\bar{B}(a,\rho)=\{x\in M: d(x,a)\leq\rho\}\] in $M$ centered at
$a$ with a finite radius $\rho>0$ and a probability measure $\mu\in
P_1(M)$ such that
\[\mu (\bar{B}(a,\rho))=\alpha>\frac{1}{2}.\]

The aim of this section is to estimate the positions of the
Fr\'{e}chet medians of $\mu$, which gives a quantitative estimation
for robustness. To this end,
the following type of functions are of fundamental importance for
our methods. Let $x,z\in M$, define
\[h_{x,z}:\quad
\bar{B}(a,\rho)\longrightarrow\mathbf{R},\quad p\longmapsto
d(x,p)-d(z,p).\] Obviously, $h_{x,z}$ is continuous and attains its
minimum.\\

Our method of estimating the position of $Q_{\mu}$ is essentially
based on the following simple observation.
\begin{proposition}\label{basic principle}
Let $x\in \bar{B}(a,\rho)^c$ and assume that there exists $z\in M$
such that
\[\min_{p\in\bar{B}(a,\rho)}
h_{x,z}(p)>\frac{1-\alpha}{\alpha}\,d(x,z),\] then $x\notin
Q_{\mu}$.
\end{proposition}

\begin{proof}
Clearly one has
\begin{align*}
f_{\mu}(x)-f_{\mu}(z)&=\int_{\bar{B}(a,\rho)}(d(x,p)-d(z,p))\mu(dp)+\int_{M\setminus\bar{B}(a,\rho)}(d(x,p)-d(z,p))\mu(dp)\notag\\
&\geq\alpha\min_{p\in\bar{B}(a,\rho)} h_{x,z}(p)-(1-\alpha)d(x,z)>0.
\end{align*}
The proof is complete.
\end{proof}

By choosing the dominating point $z=a$ in Proposition  \ref{basic
principle} we get the following basic estimation.

\begin{theorem}\label{general estimation}
The set $Q_{\mu}$ of all the Fr\'{e}chet medians of $\mu$ verifies
\[Q_{\mu}\subset \bar{B}\bigg(a,\frac{2\alpha\rho}{2\alpha-1}\bigg).\]
\end{theorem}

\begin{proof}
Observe that for every $p\in\bar{B}(a,\rho)$,
\[h_{x,a}(p)=d(x,p)-d(a,p)\geq d(x,a)-2d(a,p)\geq d(x,a)-2\rho.\]
Hence Proposition \ref{basic principle} yields
\begin{align*}
Q_{\mu}\cap \bar{B}(a,\rho)^c&\subset\big\{x\in
M:\min_{p\in\bar{B}(a,\rho)}h_{x,a}(p)\leq\frac{1-\alpha}{\alpha}\,d(x,a)\big\}\\
&\subset\big\{x\in
M:d(x,a)-2\rho\leq\frac{1-\alpha}{\alpha}\,d(x,a)\big\}\\
&=\big\{x\in M:d(x,a)\leq\frac{2\alpha\rho}{2\alpha-1}\big\}.
\end{align*}
The proof is complete.
\end{proof}

\begin{remark}
It is easily seen that the conclusions of Proposition \ref{basic
principle} and Theorem \ref{general estimation} also hold if $M$ is
only a proper metric space.
\end{remark}

\begin{remark}
As a direct corollary to Theorem \ref{general estimation}, if $\mu$
is a probability measure in $P_1(M)$ such that for some point $m\in
M$ one has $\mu\{m\}>1/2$, then $m$ is the unique Fr\'{e}chet median
of $\mu$.
\end{remark}





Thanks to Theorem \ref{general estimation}, from now on we only have
to work in the closed geodesic ball
\[B_{*}=\bar{B}\bigg(a,\frac{2\alpha\rho}{2\alpha-1}\bigg).\]
Thus let $\Delta$ be an upper bound of sectional curvatures in $B_*$
and $\inj$ be the injectivity radius of $B_*$. Moreover, we shall
always assume that the following concentration condition is
fulfilled throughout the rest part of this section:
\begin{Assumption}\label{assump}
\[\frac{2\alpha\rho}{2\alpha-1}<r_*:=\min\{\frac{\pi}{\sqrt{\Delta}}\,\,,\inj\,\},\]
where if $\Delta\leq0$, then $\pi/\sqrt{\Delta}$ is interpreted as
$+\infty$.
\end{Assumption}


In view of Proposition \ref{basic principle} and Theorem
\ref{general estimation}, estimating the position of $Q_{\mu}$ can
be achieved by estimating the minimum of the functions $h_{x,z}$ for
some $x,z\in
B_*$.  
The following lemma enables us to use the comparison argument
proposed in \cite{Afsari} to compare the configurations in $B_*$
with the ones in model spaces in order to obtain lower bounds of the
functions $h_{x,z}$.

\begin{lemma}\label{can be compared}
Let $x\in B_*\setminus\bar{B}(a,\rho)$ and $y$ be the intersection
point of the boundary of $\bar{B}(a,\rho)$ and the minimal geodesic
joining $x$ and $a$. Let $z\neq x$ be another point on the minimal
geodesic joining $x$ and $a$. Assume that $d(a,x)+d(a,z)<r_*$, then
\[\argmin h_{x,z}\subset
\{p\in\bar{B}(a,\rho):\,d(x,p)+d(p,z)+d(z,x)<2r_*\}.\]
\end{lemma}

\begin{proof}
Let $p\in\bar{B}(a,\rho)$ such that $d(x,p)+d(p,z)+d(z,x)\geq2r_*$,
then
\begin{align}\label{compare}
h_{x,z}(p)&\geq2r_*-d(x,z)-2d(z,p)\notag\\
&>2(d(a,x)+d(a,z))-d(x,z)-2(d(a,z)+\rho)\notag\\
&=d(x,y)-d(a,y)+d(a,z).
\end{align}
If $d(a,y)>d(a,z)$, then (\ref{compare}) yields
$h_{x,z}(p)>h_{x,z}(y)$, thus $p$ cannot be a minimum point of
$h_{x,z}$. On the other hand, if $d(a,y)\leq d(a,z)$, then
(\ref{compare}) gives that $h_{x,z}(p)>d(x,y)+d(y,z)\geq d(x,z)$,
which is impossible. Hence in either case, every minimum point $p$
of $h_{x,z}$ must verify $d(x,p)+d(p,z)+d(z,x)<2r_*$.
\end{proof}



As a preparation for the comparison arguments in the following, let
us recall the definition of model spaces. For a real number
$\kappa$, the
model space $\mathbb{M}^2_{\kappa}$ is defined as follows:\\
1) if $\kappa>0$, then $\mathbb{M}^2_{\kappa}$ is obtained from the
sphere $\mathbb{S}^2$ by multiplying the distance function by
$1/\sqrt{\kappa}$;\\
2) if $\kappa=0$, then $\mathbb{M}^2_{\kappa}$ is the Euclidean space $\mathbb{E}^2$;\\
3) if $\kappa<0$, then $\mathbb{M}^2_{\kappa}$ is obtained from the
hyperbolic space $\mathbb{H}^2$ by multiplying the distance function
by $1/\sqrt{-\kappa}$.\\
Moreover, the distance between two points $\bar{x}$ and $\bar{y}$ in
$\mathbb{M}^2_{\kappa}$ will be denoted by
$\bar{d}(\bar{x},\bar{y})$.

The following proposition says that for the positions of Fr\'{e}chet
medians, if comparisons can be done, then the model space
$\mathbb{M}_{\Delta}^2$ is the worst case.

\begin{proposition}\label{sphere worst}
Consider in $\mathbb{M}_{\Delta}^2$ the same configuration as that
in Lemma \ref{can be compared}: a closed geodesic ball
$\bar{B}(\bar{a},\rho)$ and a point $\bar{x}$ such that
$\bar{d}(\bar{x},\bar{a})=d(x,a)$. We denote $\bar{y}$ the
intersection point of the boundary of $\bar{B}(\bar{a},\rho)$ and
the minimal geodesic joining $\bar{x}$ and $\bar{a}$. Let $\bar{z}$
be a point in the minimal geodesic joining $\bar{x}$ and $\bar{a}$
such that $\bar{d}(\bar{a},\bar{z})=d(a,z)$. Assume that
$d(a,x)+d(a,z)<r_*$, then
\[\min_{p\in\bar{B}(a,\rho)}
h_{x,z}(p)\geq\min_{\bar{p}\in\bar{B}(\bar{a},\rho)}
\bar{h}_{\bar{x},\bar{z}}(\bar{p}),\] where
$\bar{h}_{\bar{x},\bar{z}}(\bar{p}):=\bar{d}(\bar{x},\bar{p})-\bar{d}(\bar{z},\bar{p})$.
\end{proposition}

\begin{proof}
Let $p\in\argmin h_{x,z}$. Consider a comparison point
$\bar{p}\in\mathbb{M}_{\Delta}^2$ such that
$\bar{d}(\bar{z},\bar{p})=d(z,p)$ and
$\angle\,\bar{a}\bar{z}\bar{p}=\angle\,azp$. Then the assumption
$d(a,x)+d(a,z)<r_*$ and the hinge version of Alexandrov-Toponogov
comparison theorem (see \cite[Exercise IX.1, p. 420]{Chavel}) yield
that $\bar{d}(\bar{a},\bar{p})\leq d(a,p)=\rho$, i.e.
$\bar{p}\in\bar{B}(\bar{a},\rho)$. Now by hinge comparison again and
Lemma \ref{can be compared}, we get $\bar{d}(\bar{p},\bar{x})\leq
d(p,x)$, which implies that
\[h_{x,z}(p)\geq\bar{h}_{\bar{x},\bar{z}}(\bar{p})\geq\min_{\bar{p}\in\bar{B}(\bar{a},\rho)}
\bar{h}_{\bar{x},\bar{z}}(\bar{p}).\] The proof is complete.
\end{proof}


According to Proposition \ref{sphere worst}, it suffices to find the
minima of the functions $h_{x,z}$ when $M$ equals $\mathbb{S}^2$,
$\mathbb{E}^2$ and $\mathbb{H}^2$, which are of constant curvatures
$1$, $0$ and $-1$, respectively.

\begin{proposition}\label{min h}
Let $t,u\geq0$ such that
$u<\rho+t\leq 2\alpha\rho/(2\alpha-1)$.\\

i) If  $M=\mathbb{S}^2$, let $x=(\sin(\rho+t),0,\cos(\rho+t))$ and
$z=(\sin u,0,\cos u)$. Assume that $\rho+t+u<\pi$, then
\[
\min_{\bar{B}(a,\rho)} h_{x,z}=
\begin{cases}
t-\rho+u,\qquad\qquad\qquad\qquad\quad\text{if\quad $\cot
u\geq2\cot\rho-\cot(\rho+t)$;}\\\\
\arccos\bigg(\cos(\rho+t-u)+\cfrac{\sin^2\rho\sin^2(\rho+t-u)}{2\sin
u\sin(\rho+t)}\bigg),~\quad\quad \text{if\quad not.}
\end{cases}
\]

ii) If $M=\mathbb{E}^2$, let $a=(0,0)$, $x=(\rho+t,0)$, $z=(u,0)$,
then
\[
\min_{\bar{B}(a,\rho)} h_{x,z}=
\begin{cases}
t-\rho+u,\qquad\qquad\qquad\qquad\qquad\qquad\text{if\quad$u\leq\cfrac{(\rho+t)\rho}{\rho+2t}$;}\\\\
(\rho+t-u)\sqrt{1-\cfrac{\rho^2}{u(\rho+t)}},\qquad\qquad\qquad\qquad\quad\text{if\quad
not.}
\end{cases}
\]

iii) If $M=\mathbb{H}^2$, let $a=(0,0,1)$,
$x=(\sinh(\rho+t),0,\cosh(\rho+t))$ and $z=(\sinh u,0,\cosh u)$,
then
\[
\min_{\bar{B}(a,\rho)} h_{x,z}=
\begin{cases}
t-\rho+u, \qquad\qquad\qquad\qquad\text{if\quad$\coth
u\geq2\coth\rho-\coth(\rho+t)$;}\\\\
\arccosh\bigg(\cosh(\rho+t-u)-\cfrac{\sinh^2\rho\sinh^2(\rho+t-u)}{2\sinh
u\sinh(\rho+t)}\bigg),\quad\text{if\quad not.}
\end{cases}
\]
\end{proposition}

We shall only prove the result for the case when $M=\mathbb{S}^2$,
since the proofs for $M=\mathbb{E}^2$  and $M=\mathbb{H}^2$ are
similar and easier. The proof consists of some lemmas, the first one
below says that $h_{x,z}$ is smooth at its minimum points which can
only appear on the boundary of the ball $\bar{B}(a,\rho)$.

\begin{lemma}\label{differentiable}
Let $x^\prime$ and $z^\prime$ be the antipodes of $x$ and $z$. Then
$z^\prime\notin\bar{B}(a,\rho)$ and all the local minimum points of
$h_{x,z}$ are contained in
$\partial\bar{B}(a,\rho)\setminus\{x^\prime\}$.
\end{lemma}

\begin{proof}
It is easily seen that $d(z^\prime,a)=\pi-u>\rho$, so that
$z^\prime\notin\bar{B}(a,\rho)$. Observe that $x^{\prime}$ is a
global maximum point of $h_{x,z}$ which is not locally constant, so
that $x^{\prime}$ cannot be a local minimum. Now let $p\in
B(a,\rho)$ be a local minimum of $h_{x,z}$, then $h_{x,z}$ is
smooth at $p$. It follows that $\grad h_{x,z}(p)=0$, 
which yields that $h_{x,z}(p)=d(x,z)$, this is a contradiction. The
proof is complete.
\end{proof}


The following lemma characterizes the global minimum points of
$h_{x,z}$.

\begin{lemma}\label{argmin h}
The set of global minimum points of $h_{x,z}$ verifies
\[
\argmin h_{x,z}=
\begin{cases}
\{y\},\quad\quad\quad\quad\quad\quad\quad\quad\quad\quad\quad\text{if $\cot u\geq2\cot\rho-\cot(\rho+t)$;}\\
\{\,p\in\partial\bar{B}(a,\rho):\,\cfrac{\sin(\rho+t)}{\sin
d(x,p)}=\cfrac{\sin u}{\sin
d(z,p)}\,\},\quad\quad\quad\quad\quad\text{if\,\,not,}
\end{cases}
\]
where $y$ is the intersection point of the boundary of
$\bar{B}(a,\rho)$ and the minimal geodesic joining $x$ and $a$.
\end{lemma}

\begin{proof}
Thanks to Lemma \ref{differentiable}, it suffices to find the global
minimum points of $h_{x,z}$ for
$p=(\sin\rho\cos\theta,\sin\rho\sin\theta,\cos\rho)$ and $\theta\in
[0,2\pi)$. In this case,
\begin{align*}
h_{x,z}(p)=&\,d(x,p)-d(z,p)\\
=&\,\arccos(\sin(\rho+t)\sin\rho\cos\theta+\cos(\rho+t)\cos\rho)\\
&\,-\arccos(\sin u\sin\rho\cos\theta+\cos u\cos\rho)\\
:= &\,h(\theta).
\end{align*}
Hence let $p=(\sin\rho\cos\theta,\sin\rho\sin\theta,\cos\rho)$ be a
local minimum point of $h_{x,z}$, then Lemma \ref{differentiable}
yields that $h^\prime(\theta)$ exists and equals zero. On the other
hand, by elementary calculation,
\begin{align*}
h^\prime(\theta)=\sin\rho\sin\theta
\bigg(&\frac{\sin(\rho+t)}{\sqrt{1-(\sin(\rho+t)\sin\rho\cos\theta+\cos(\rho+t)\cos\rho)^2}}\\
-&\frac{\sin u}{\sqrt{1-(\sin u\sin\rho\cos\theta+\cos
u\cos\rho)^2}}\bigg)\\
=\sin\rho\sin\theta\bigg(&\frac{\sin(\rho+t)}{\sin
d(x,p)}-\frac{\sin u}{\sin d(z,p)}\bigg).
\end{align*}
Hence we have necessarily
\[\theta=0,\,\,\pi\quad\text{or}\quad\frac{\sin(\rho+t)}{\sin
d(x,p)}=\frac{\sin u}{\sin d(z,p)}.\]

Firstly, we observe that $w$, the corresponding point $p$ when
$\theta=\pi$, cannot be a minimum point. In fact, let $w^\prime$ be
the antipode of $w$. If $d(x,a)<d(w^\prime,a)$, then
$h_{x,z}(w)=d(x,z)$. So that $w$ is a maximum point. On the other
hand, if $d(x,a)\geq d(w^\prime,a)$, then
$d(w,x)+d(x,z)+d(z,w)\equiv2\pi$. Hence Lemma \ref{can be compared}
and the condition $\rho+t+u<\pi$ imply that $w$ is not a minimum
point. So that the assertion holds.

Now assume that $p\neq w,y$ such that
\begin{equation}\label{min point not y}
\frac{\sin(\rho+t)}{\sin d(x,p)}=\frac{\sin u}{\sin d(z,p)}.
\end{equation}
Let $\beta=\angle\,zpa$. Then by the spherical law of sines,
\eqref{min point not y} is equivalent to
$\sin(\beta+\angle\,zpx)=\sin\beta$, i.e. that
$\angle\,zpx=\pi-2\beta$. Applying the spherical law of sines to
$\triangle zpx$ we get
\begin{equation}\label{ls zpx}
\frac{\sin\angle\,x}{\sin d(z,p)}=\frac{\sin2\beta}{\sin(\rho+t-u)}.
\end{equation}
By the spherical law of sines in $\triangle apx$,
\begin{equation}\label{ls apx}
\frac{\sin\angle\,x}{\sin \rho}=\frac{\sin\beta}{\sin(\rho+t)}.
\end{equation}
Then \eqref{ls zpx}/\eqref{ls apx} gives that
\begin{equation}\label{5}
\sin d(z,p)\cos\beta= \frac{\sin\rho\sin(\rho+t-u)}{2\sin (\rho+t)}.
\end{equation}
By the spherical law of cosines in $\triangle azp$,
\begin{equation}\label{6}
\sin d(z,p)\cos\beta=\frac{\cos u-\cos\rho\cos d(z,p)}{\sin\rho}.
\end{equation}
Then \eqref{5} and \eqref{6} give that
\begin{equation}\label{7}
\cos d(z,p)=\frac{2\cos
u\sin(\rho+t)-\sin^2\rho\sin(\rho+t-u)}{2\cos\rho\sin(\rho+t)}.
\end{equation}
Moreover, by \eqref{7} and spherical law of cosines in $\triangle
azp$,
\begin{equation}\label{8}
\cos\theta=\frac{\tan\rho}{2}(\cot u+\cot(\rho+t)).
\end{equation}
The condition $0<\rho+t+u<\pi$ and \eqref{8} give that
\begin{equation}\label{9}
\cos\theta=\frac{\tan\rho\sin(\rho+t+u)}{2\sin u\sin(\rho+t)}>0.
\end{equation}
Furthermore, considering $p\neq y$ we must have $\cos\theta<1$. By
\eqref{8} this is equivalent to
\begin{equation}\label{8.1}
\cot u<2\cot\rho-\cot(\rho+t),
\end{equation}
which is also equivalent to
\begin{equation}\label{10}
\frac{\sin (\rho-u)}{\sin u}<\frac{\sin t}{\sin(\rho+t)}.
\end{equation}

For the case when $u\geq\rho$ it is easily seen that $y$ is a
maximum point of $h_{x,z}$ and hence $p$ must verify \eqref{min
point not y} and the corresponding $\theta$ is determined by
\eqref{8}. Hence in this case, there are exactly two local minimum
points of $h_{x,z}$ and obviously they are also global ones.

Now let $u<\rho$, then easy computation gives
\begin{equation}\label{11}
h^{\prime\prime}(0)=\sin\rho\bigg(\frac{\sin (\rho+t)}{\sin
t}-\frac{\sin u}{\sin(\rho-u)}\bigg).
\end{equation}
Hence if $\frac{\sin (\rho+t)}{\sin t}\geq\frac{\sin
u}{\sin(\rho-u)}$, then \eqref{10} yields that $y$ is the unique
global minimum of $h_{x,z}$. In the opposite case, \eqref{11}
implies that $y$ is a local maximum point. Hence the same argument
as in the case when $u\geq\rho$ completes the proof of lemma.
\end{proof}

We need the following technical lemma.

\begin{lemma}\label{h is positive}
If $\cot u<2\cot\rho-\cot(\rho+t)$, then $0<d(x,p)-d(z,p)<\pi$ for
every $p\in\argmin h_{x,z}$.
\end{lemma}

\begin{proof}
It suffices to show $d(x,p)>d(z,p)$. For the case when $u<\rho$, we
firstly show that $d(z,p)<\pi/2$. In fact, by \eqref{7} this is
equivalent to show that
\begin{equation}\label{12}
(1+\cos^2\rho)\cos u\sin(\rho+t)+\sin^2\rho\sin u\cos(\rho+t)>0
\end{equation}
If $\rho+t\leq\pi/2$, \eqref{12} is trivially true. Now assume
$\rho+t>\pi/2$. So that $\pi/2<\rho+t<\pi-u$, which implies
$\sin(\rho+t)>\sin u$ and $\cos(\rho+t)>-\cos u$. Hence we get
\begin{align*}
&(1+\cos^2\rho)\cos u\sin(\rho+t)+\sin^2\rho\sin u\cos(\rho+t)\\
>&(1+\cos^2\rho)\cos u\sin u-\sin^2\rho\sin u\cos u\\
=&2\cos^2\rho\cos u\sin u>0.
\end{align*}
So that $d(z,p)<\pi/2$ holds. Now if $d(x,p)\geq\pi/2$, then
obviously $d(x,p)>d(z,p)$. So that assume $d(x,p)<\pi/2$. Observe
that $\rho+t+u<\pi$ implies $\sin u<\sin(\rho+t)$, then \eqref{min
point not y} yields $d(x,p)>d(z,p)$.

For the case when $u\geq \rho$, it suffices to show that $\cos
d(z,p)>\cos d(x,p)$ for every
$p=(\sin\rho\cos\theta,\sin\rho\sin\theta,\cos\rho)$ with
$\theta\in[0,\pi]$. Now let
\begin{align*}
g(\theta)=&\,\sin\rho\cos\theta(\sin u-\sin(\rho+t))+\cos\rho(\cos
u-\cos(\rho+t))\\
=&\,\cos d(z,p)-\cos d(x,p)
\end{align*}
Then $g^{\prime}(\theta)=-\sin\rho\sin\theta(\sin u-\sin(\rho+t)).$
Observe that $\rho+t+u<\pi$ and $u<\rho+t$ imply that $\sin
u<\sin(\rho+t)$, hence $g(\theta)\geq g(0)=\cos d(z,y)-\cos
d(x,y)>0$. The proof is complete.
\end{proof}

\begin{proof}[Proof of Proposition \ref{min h}]
By Lemma \ref{argmin h}, it suffices to consider the case when $\cot
u<2\cot\rho-\cot(\rho+t)$. Let $p\in\argmin h_{x,z}$, then by
\eqref{8} and the spherical law of cosines in $\triangle apx$,
\begin{equation}\label{13}
\cos d(x,p)=\frac{2\cos (\rho+t)\sin
u+\sin^2\rho\sin(\rho+t-u)}{2\cos\rho\sin u}.
\end{equation}
Now let $u=\rho-v$, then $\rho+t-u=t+v$. So that \eqref{7} and
\eqref{13} become
\begin{equation}\label{a7}
\cos d(z,p)=\frac{2\cos
(\rho-v)\sin(\rho+t)-\sin^2\rho\sin(t+v)}{2\cos\rho\sin(\rho+t)}.
\end{equation}
\begin{equation}\label{a13}
\cos d(x,p)=\frac{2\cos (\rho+t)\sin(\rho-v)
+\sin^2\rho\sin(t+v)}{2\cos\rho\sin (\rho-v)}.
\end{equation}
It follows that
\begin{align}\label{a14}
\cos d(z,p)\cos d(x,p)
=&\,(4\cos(\rho-v)\sin(\rho+t)\cos(\rho+t)\sin(\rho-v)\notag\\
&\,+2\sin^2\rho\sin^2(t+v)-\sin^4\rho\sin^2(t+v))\notag\\
&\,/(4\cos^2\rho\sin(\rho+t)\sin(\rho-v)).
\end{align}
On the other hand, \eqref{min point not y} and \eqref{a7} yield that
\begin{align}\label{a15}
\sin d(z,p)\sin d(x,p)
=&\,(1-\cos^2d(z,p))\frac{\sin(\rho+t)}{\sin(\rho-v)}\notag\\
=&\,(4\sin^2(\rho+t)(\cos^2\rho-\cos^2(\rho-v))-\sin^4\rho\sin^2(t+v)\notag\\
&\,+4\sin^2\rho\sin(t+v)\sin(\rho+t)\cos(\rho-v))\notag\\
&\,/(4\cos^2\rho\sin(\rho+t)\sin(\rho-v)).
\end{align}
Then by \eqref{a14} and \eqref{a15} we obtain
\begin{align}
&4\cos^2\rho\sin(\rho+t)\sin(\rho-v)(\cos(d(x,p)-d(z,p))-\cos(t+v))\notag\\
=&\,4\cos^2\rho\sin(\rho+t)\sin(\rho-v)(\cos d(x,p)\cos d(z,p)+\sin d(x,p)\sin d(z,p)-\cos(t+v))\notag\\
=&\,4\cos(\rho-v)\sin(\rho-v)\cos(\rho+t)\sin(\rho+t)+2\sin^2\rho\cos^2\rho\sin^2(t+v)\notag\\
&+4\sin^2(\rho+t)(\cos^2\rho-\cos^2(\rho-v))+4\sin^2\rho\sin(\rho+t)\cos(\rho-v)\sin(t+v))\notag\\
&-4\cos^2\rho\sin(\rho+t)\sin(\rho-v)\cos(t+v)\notag\\
=&\,(-4\cos^4\rho\cos^2v\sin^2t-4\cos^4\rho\sin^2t\sin^2v+4\cos^4\rho\sin^2t)\notag\\
&\,+(-8\cos^3\rho\cos t\cos^2v\sin\rho\sin t-8\cos^3\rho\cos
t\sin\rho\sin t\sin^2v+8\cos^3\rho\cos t\sin\rho\sin t)\notag\\
&\,+(-4\cos^2\rho\cos^2t\cos^2v\sin^2\rho-2\cos^2\rho\cos^2t\sin^2\rho\sin^2v+4\cos^2\rho\cos^2t\sin^2\rho)\notag\\
&\,+4\cos^2\rho\sin^2\rho\sin
v\cos v\sin t\cos t+2\cos^2\rho\cos^2v\sin^2\rho\sin^2t\notag\\
=&\,2\cos^2\rho\sin^2\rho\cos^2t\sin^2v+4\cos^2\rho\sin^2\rho\sin
v\cos v\sin t\cos t+2\cos^2\rho\cos^2v\sin^2\rho\sin^2t\notag\\
=&\,2\cos^2\rho\sin^2\rho\sin^2(t+v).\notag
\end{align}
As a result,
\begin{align*}
\cos(d(x,p)-d(z,p))
=&\,\cos(t+v)+\frac{2\cos^2\rho\sin^2\rho\sin^2(t+v)}{4\cos^2\rho\sin(\rho+t)\sin(\rho-v)}\\
=&\,\cos(t+v)+\frac{\sin^2\rho\sin^2(t+v)}{2\sin(\rho+t)\sin(\rho-v)}\\
=&\,\cos(\rho+t-u)+\frac{\sin^2\rho\sin^2(\rho+t-u)}{2\sin
u\sin(\rho+t)}.
\end{align*}
Now it suffices to use Lemma \ref{h is positive} to finish the
proof.
\end{proof}

We also need the following lemma.

\begin{lemma}\label{F}
Let $\kappa$ be real number and $1/2<\alpha\leq1$. For
$t\in(0,\rho/(2\alpha-1)]$ define
$$
F_{\alpha,\rho,\kappa}(t)=
\begin{cases}
\cot(\sqrt{\kappa}(2\alpha-1)t)-\cot(\sqrt{\kappa} t)-2\cot(\sqrt{\kappa}\rho),& \text{if $\kappa>0$;}\\
(1-\alpha)\rho-(2\alpha-1)t,& \text{if $\kappa=0$;}\\
\coth(\sqrt{-\kappa}(2\alpha-1)t)-\coth(\sqrt{-\kappa}
t)-2\coth(\sqrt{-\kappa}\rho),& \text{if $\kappa<0$.}
\end{cases}
$$
Assume that $1/2<\alpha<1$, then there exists a unique
$t_{\kappa}\in(0,\rho/(2\alpha-1))$ such that
$$\bigg\{t\in(0,\frac{\rho}{2\alpha-1}]:\,\,F_{\alpha,\rho,\kappa}(t)\geq0\bigg\}=(0,t_{\kappa}].$$
In this case, when $\kappa\leq0$, the function
$F_{\alpha,\rho,\kappa}$ is strictly deceasing.
\end{lemma}

\begin{proof}
We only prove the case when $\kappa=1$, since the proof of the other
two cases are similar and easier. Observe that
$F_{\alpha,\rho,1}(0+)=+\infty$ (since $1/2<\alpha<1$) and
$F_{\alpha,\rho,1}(\rho/(2\alpha-1))<0$ (since
$2\alpha\rho/(2\alpha-1)<\pi$), then there exists some
$t_1\in(0,\rho/(2\alpha-1))$ such that $F_{\alpha,\rho,1}(t_1)=0$.
Moreover,
$$F^\prime_{\alpha,\rho,1}(t)=\frac{1}{\sin^2((2\alpha-1)t)}\bigg(\big(\frac{\sin((2\alpha-1)t)}{\sin
t}\big)^2-(2\alpha-1)\bigg).$$ Observe that the function
$l(t)=\sin((2\alpha-1)t)/\sin t$ is strictly increasing on
$(0,\pi/(2\alpha)]$, $l^2(0+)=(2\alpha-1)^2<2\alpha-1$ and
$l^2(\pi/(2\alpha))=1>2\alpha-1$. Hence there exists a unique
$s\in(0,\pi/(2\alpha))$ such that if $t<s$, then
$F^\prime_{\alpha,\rho,1}(t)<0$; if $t=s$, then
$F^\prime_{\alpha,\rho,1}(t)=0$; if $t>s$, then
$F^\prime_{\alpha,\rho,1}(t)>0$. Hence $F_{\alpha,\rho,1}$ is
strictly decreasing on $(0,s]$ and strictly increasing on
$[s,\pi/(2\alpha)]$. Since $\rho/(2\alpha-1)<\pi/(2\alpha)$ and
$F_{\alpha,\rho,1}(\rho/(2\alpha-1))<0$, the point $t_1$ must be
unique. Moreover, it is easily seen that
$\{F_{\alpha,\rho,1}\geq0\}=(0,t_1]$. The proof is complete.
\end{proof}

The main theorem of this section is justified by the lemma below.

\begin{lemma}\label{of sense}
Assumption \ref{assump} implies that
$$\frac{\alpha
S_{\Delta}(\rho)}{\sqrt{2\alpha-1}}<
S_{\Delta}(\frac{r_*}{2}),\,\,\text{where}\,\, S_{\Delta}(t):=
\begin{cases}
\sin(\sqrt{\Delta}\,t),& \text{if $\Delta>0$;}\\
t,& \text{if $\Delta=0$;}\\
\sinh(\sqrt{-\Delta}\,t),& \text{if $\Delta<0$.}
\end{cases}
$$

\end{lemma}
\begin{proof}
We only prove the case when $\Delta>0$, since the proof for the
cases when $\Delta\leq0$ are easier. Without loss of generality, we
can assume that $\Delta=1$. Since
$(2\alpha-1)r_*/(2\alpha)\leq\pi/2$, we have
$$\frac{2\alpha\rho}{2\alpha-1}<r_*\Longleftrightarrow\sin\rho<\sin\frac{2\alpha-1}{2\alpha}r_*.$$
So that it is sufficient to show
$$\frac{\alpha}{\sqrt{2\alpha-1}}\sin\frac{2\alpha-1}{2\alpha}r_*<\sin\frac{r_*}{2}.$$
To this end, let $c=r_*/2\in(0,\pi/2]$, we will show that the
function
$$f(\alpha)=\frac{\alpha}{\sqrt{2\alpha-1}}\sin\frac{2\alpha-1}{\alpha}c$$
is strictly increasing for $\alpha\in(1/2,1)$. Easy computation
gives that
$$f^{\prime}(\alpha)>0\Longleftrightarrow\frac{\tan(\theta c)}{\theta c}<\frac{2-\theta}{1-\theta},$$
where $\theta=(2\alpha-1)/\alpha\in(0,1)$. Observe that the function
$x\longmapsto \tan x/x$ is increasing on $[0,\pi/2)$,  hence it
suffices to show that
$$\frac{\tan(\theta\pi/2)}{(\theta\pi/2)}<\frac{2-\theta}{1-\theta}.$$
This is true because Becker-Stark inequality (see \cite{Becker})
yields
$$\frac{\tan(\theta\pi/2)}{(\theta\pi/2)}<\frac{1}{1-\theta^2}<\frac{2-\theta}{1-\theta}.$$
The proof is complete.
\end{proof}

Now we are ready to give the main result of this section.

\begin{theorem}\label{position}
The following estimations hold:\\

i) If $\Delta>0$ and $Q_{\mu}\subset \bar{B}(a,r_*/2)$, then
$$
Q_{\mu}\subset\bar{B}\bigg(a,\frac{1}{\sqrt{\Delta}}\arcsin\big(\frac{\alpha\sin(\sqrt{\Delta}\rho)}{\sqrt{2\alpha-1}}\big)\bigg).
$$
Moreover, any of the two conditions below implies $Q_{\mu}\subset
\bar{B}(a,r_*/2)$:

$$a)\quad\frac{2\alpha\rho}{2\alpha-1}\leq\frac{r_*}{2};
\qquad b)\quad
\frac{2\alpha\rho}{2\alpha-1}>\frac{r_*}{2}\quad\text{and}\quad
F_{\alpha,\rho,\Delta}(\frac{r_*}{2}-\rho)\leq0.$$

ii) If $\Delta=0$, then
$$Q_{\mu}\subset\bar{B}\bigg(a,\frac{\alpha\rho}{\sqrt{2\alpha-1}}\bigg).
$$

iii) If $\Delta<0$, then
$$Q_{\mu}\subset\bar{B}\bigg(a,\frac{1}{\sqrt{-\Delta}}\arcsinh\big(\frac{\alpha\sinh(\sqrt{-\Delta}\rho)}{\sqrt{2\alpha-1}}\big)\bigg).
$$
Finally, Lemma \ref{of sense} ensures that any of the above three
closed balls is contained in the open ball $B(a,r_*/2)$.
\end{theorem}

\begin{proof}
Firstly, we consider the case when $\Delta>0$. Without loss of
generality, we can assume that $\Delta=1$. For every $x\in
B_*\setminus \bar{B}(a,\rho)$, let
$t_x=d(a,x)-\rho\in(0,\rho/(2\alpha-1)]$. By Propositions \ref{basic
principle} and \ref{sphere worst}, if there exists some $z$ on the
minimal geodesic joining $x$ and $a$ such that
$u_z=d(a,z)\in[0,\rho+t_x)$ verifies $\rho+t_x+u_z<r_*$ and
$\min_{\bar{B}(\bar{a},\rho)}\bar{h}_{\bar{x},\bar{z}}>(1-\alpha)(\rho+t_x-u_z)/\alpha$,
then $x\notin Q_{\mu}$.
Or equivalently,
\begin{align*}
&Q_{\mu}\cap\bar{B}(a,\rho)^c\\
\subset&\,\bigg\{x\in B_*\setminus
\bar{B}(a,\rho):\,\,t_x\in(0,\frac{\rho}{2\alpha-1}]\,\,\text{has
the property that for
every}\,\, u_z\in[0,\rho+t_x)\\
&\phantom{x\in B_*\setminus \bar{B}(a,\rho):\,\,+{}} \text{such
that}\,\,\rho+t_x+u_z<r_*,\,
\min\bar{h}_{\bar{x},\bar{z}}\leq\frac{1-\alpha}{\alpha}(\rho+t_x-u_z)\bigg\}:=A.
\end{align*} Since the restrictive condition of the set $A$ is
only on $t_x$, for simplicity and without ambiguity, by dropping the
subscripts of $t_x$ and $u_z$ we rewrite $A$ in the following form:
\begin{align*}
&\,\bigg\{t\in(0,\frac{\rho}{2\alpha-1}]:\,\,\text{for
every}\,\, u\in[0,\rho+t)\,\,\text{such that}\,\,\rho+t+u<
r_*,\\
&\phantom{t\in(0,\frac{\rho}{2\alpha-1}]:+{}}\min
\bar{h}_{\bar{x},\bar{z}}\leq\frac{1-\alpha}{\alpha}(\rho+t-u)\bigg\}\\
=&\,\,\bigg\{t\in(0,\frac{r_*}{2}-\rho]:\,\,\text{for every}\,\,
u\in[0,\rho+t)\,\,\text{such that}\,\,\rho+t+u<
r_*,\\
&\phantom{t\in(0,\frac{\rho}{2\alpha-1}]:\,\,+{}}\min
\bar{h}_{\bar{x},\bar{z}}\leq\frac{1-\alpha}{\alpha}(\rho+t-u)\bigg\}\\
&\cup\bigg\{t\in(\frac{r_*}{2}-\rho,\frac{\rho}{2\alpha-1}]:\,\,\text{for
every}\,\, u\in[0,\rho+t)\,\,\text{such that}\,\,\rho+t+u<
r_*,\\
&\phantom{\cup\bigg\{t\in(\frac{r_*}{2}-\rho,\frac{\rho}{2\alpha-1}]:+{}}\min
\bar{h}_{\bar{x},\bar{z}}\leq\frac{1-\alpha}{\alpha}(\rho+t-u)\bigg\}
:=\, B\cup C.
\end{align*}
Observe that for $t\in(0,r_*/2-\rho\,]$ and $u\in[0,\rho+t)$, we
always have $\rho+t+u<r_*$, hence by Proposition \ref{min h} and
Lemma \ref{h is positive},
\begin{align*}
B=&\,\bigg\{t\in(0,\frac{r_*}{2}-\rho]:\,\,\text{for every}\,\,
u\in[0,\rho+t),\,\,\min
\bar{h}_{\bar{x},\bar{z}}\leq\frac{1-\alpha}{\alpha}(\rho+t-u)\bigg\}\\
=&\,\bigg\{t\in(0,\frac{r_*}{2}-\rho]:\,\,\text{for every}\,\,
u\in[0,\rho+t)\,\,\text{such that}\,\,\cot u\geq
2\cot\rho-\cot(\rho+t),\\
&\phantom{t\in(0,\frac{r_*}{2}-\rho]:\,\,\,+{}}
t-\rho+u\leq\frac{1-\alpha}{\alpha}(\rho+t-u)\bigg\}\\
&\cap\bigg\{t\in(0,\frac{r_*}{2}-\rho]:\,\,\text{for every}\,\,
u\in[0,\rho+t)\,\,\text{such that}\,\,\cot u<
2\cot\rho-\cot(\rho+t),\\
&\phantom{\cap\bigg\{t\in(0,\frac{r_*}{2}-\rho]:+{}}
\cos(\rho+t-u)+\frac{\sin^2\rho\sin^2(\rho+t-u)}{2\sin
u\sin(\rho+t)}\geq\cos\big(\frac{1-\alpha}{\alpha}(\rho+t-u)\big)\bigg\}\\
=&\,\bigg\{t\in(0,\frac{r_*}{2}-\rho]:\,\,\text{for every}\,\,
u\in[0,\rho+t)\,\,\text{such that}\,\,\cot u\geq
2\cot\rho-\cot(\rho+t),\\
&\phantom{t\in(0,\frac{r_*}{2}-\rho]:\,\,\,+{}}
u\leq\rho-(2\alpha-1)t\bigg\}\\
&\cap\bigg\{t\in(0,\frac{r_*}{2}-\rho]:\,\,\text{for every}\,\,
u\in[0,\rho+t)\,\,\text{such that}\,\,\cot u<
2\cot\rho-\cot(\rho+t),\\
&\phantom{\cap\bigg\{t\in(0,\frac{r_*}{2}-\rho]:+{}}
\sin(\rho+t)\leq\frac{\sin^2\rho}{4\sin
u}\frac{\sin(\rho+t-u)}{\sin\cfrac{\rho+t-u}{2\alpha}}\frac{\sin(\rho+t-u)}{\sin\big(\cfrac{2\alpha-1}{2\alpha}(\rho+t-u)\big)}\bigg\}\\
=&\,\bigg\{t\in(0,\frac{r_*}{2}-\rho]:\,\,\cot(\rho-(2\alpha-1)t)\leq
2\cot\rho-\cot(\rho+t)\bigg\}\\
\end{align*}
\begin{align}\label{A sense}
&\cap\bigg\{t\in(0,\frac{r_*}{2}-\rho]:\,\,\text{for every}\,\,
u\in[0,\rho+t)\,\,\text{such that}\,\,\cot u<
2\cot\rho-\cot(\rho+t),\notag\\
&\phantom{\cap\bigg\{t\in(0,\frac{r_*}{2}-\rho]:+{}}
\sin(\rho+t)\leq\frac{\sin^2\rho}{4\sin
u}\frac{\sin(\rho+t-u)}{\sin\cfrac{\rho+t-u}{2\alpha}}\frac{\sin(\rho+t-u)}{\sin\big(\cfrac{2\alpha-1}{2\alpha}(\rho+t-u)\big)}\bigg\}\notag\\
\subset&\bigg\{t\in(0,\frac{r_*}{2}-\rho]:\,\,\text{for every}\,\,
u\in(\rho-(2\alpha-1)t,\rho+t),\notag\\
&\phantom{t\in(0,\frac{r_*}{2}-\rho]:\,\,+{}}
\sin(\rho+t)\leq\frac{\sin^2\rho}{4\sin u}\frac{\sin(\rho+t-u)}{\sin\cfrac{\rho+t-u}{2\alpha}}\frac{\sin(\rho+t-u)}{\sin\big(\cfrac{2\alpha-1}{2\alpha}(\rho+t-u)\big)}\bigg\}\notag\\
\subset&\bigg\{t\in(0,\frac{r_*}{2}-\rho]:\,\,\sin(\rho+t)\leq\frac{\sin^2\rho}{4\sin(\rho+t)}\cdot2\alpha\cdot\frac{2\alpha}{2\alpha-1}\bigg\}\notag\\
=&\,\bigg\{t\in(0,\frac{r_*}{2}-\rho]:\,\,\rho+t\leq\arcsin\big(\frac{\alpha\sin\rho}{\sqrt{2\alpha-1}}\big)\bigg\}.
\end{align}
Hence if $Q_{\mu}\subset \bar{B}(a,r_*/2)$, then $C=\phi$ and
\eqref{A sense} says that
$$Q_{\mu}\subset\bar{B}\bigg(a,\arcsin\big(\frac{\alpha\sin\rho}{\sqrt{2\alpha-1}}\big)\bigg),$$
this completes the proof of the first assertion of $i)$. To show the
second one, observe that if $a)$ holds, then Theorem \ref{general
estimation} implies the desired result. Hence assume that $b)$
holds. By Proposition \ref{min h} one has
\begin{align*}
C\subset&\bigg\{t\in(\frac{r_*}{2}-\rho,\frac{\rho}{2\alpha-1}]:\,\,\text{for
every}\,\, u\in[0,r_*-(\rho+t))\,\,\text{such that}\notag\\
&\phantom{t\in(\frac{r_*}{2}-\rho,\frac{\rho}{2\alpha-1}]:\,\,+{}}\cot
u\geq 2\cot\rho-\cot(\rho+t),\notag\\
&\phantom{t\in(\frac{r_*}{2}-\rho,\frac{\rho}{2\alpha-1}]:\,\,+{}}
t-\rho+u\leq\frac{1-\alpha}{\alpha}(\rho+t-u)\bigg\}\notag\\
=&\bigg\{t\in(\frac{r_*}{2}-\rho,\frac{\rho}{2\alpha-1}]:\,\,\cot(\rho-(2\alpha-1)t)\leq
2\cot\rho-\cot(\rho+t)\bigg\}.
\end{align*}
Observe that for $t\in(r_*/2-\rho,\,\rho/(2\alpha-1))$ we have
\begin{align*}
&\cot(\rho-(2\alpha-1)t)\leq 2\cot\rho-\cot(\rho+t)\\
\Longleftrightarrow\,&\cot(\rho-(2\alpha-1)t)-\cot\rho\leq \cot\rho-\cot(\rho+t)\\
\Longleftrightarrow\,&\frac{\sin(\rho-(2\alpha-1)t)}{\sin((2\alpha-1)t)}\geq\frac{\sin(\rho+t)}{\sin
t}\\
\Longleftrightarrow\,&\cot((2\alpha-1)t)-\cot\rho\geq\cot\rho+\cot t\\
\Longleftrightarrow\,&F_{\alpha,\rho,1}(t)\geq0.
\end{align*}
Hence
\begin{align*}
C\subset\bigg\{t\in(\frac{r_*}{2}-\rho,\frac{\rho}{2\alpha-1}]:\,\,F_{\alpha,\rho,1}(t)\geq0\bigg\}:=D.
\end{align*}
If $\alpha=1$, clearly $D=\phi$. Now let $1/2<\alpha<1$, then Lemma
\ref{F} yields that
\begin{align*}
D=(\frac{r_*}{2}-\rho,\frac{\rho}{2\alpha-1}]\cap(0,t_1]=\phi.
\end{align*}
Thus $C=\phi$ still holds, that is, $Q_{\mu}\subset
\bar{B}(a,r_*/2)$. The proof of $i)$ is complete.

Now let us turn to the proof of $ii)$ and $iii)$. In fact, the proof
for these two cases are essentially the same as that of $i)$ except
to note that we no longer need to assume that $Q_{\mu}\subset
\bar{B}(a,r_*/2)$, because this is implied by Assumption
\ref{assump}. To see this, if $4\alpha\rho\leq(2\alpha-1)r_*$, then
it suffices to use Theorem \ref{general estimation}. So that let us
assume $4\alpha\rho>(2\alpha-1)r_*$ and  show that
$F_{\alpha,\rho,\Delta}(r_*/2-\rho)\leq0$ for $\Delta\in\{-1,0\}$.
This is trivial if $\Delta=0$ or $\alpha=1$, hence let $\Delta=-1$
and $\alpha\in(1/2,1)$. Since
$r_*/2-\rho>(1-\alpha)\rho/(2\alpha-1)$ and $F_{\alpha,\rho,-1}$ is
strictly decreasing, it suffices to show that
$F_{\alpha,\rho,-1}((1-\alpha)\rho/(2\alpha-1))\leq0$. To this end,
define $f(\alpha)=F_{\alpha,\rho,-1}((1-\alpha)\rho/(2\alpha-1))$,
easy computation gives that
$$f^{\prime}(\alpha)=\frac{\rho}{\sinh^2((1-\alpha)\rho)}
-\frac{\rho}{(2\alpha-1)^2\sinh^2(\frac{1-\alpha}{2\alpha-1}\rho)}>0,$$
because the function $x\longmapsto\sinh x/x$ is strictly increasing.
Hence $f(\alpha)< f(1-)=2(\rho^{-1}-\coth\rho)<0$. The proof is
complete.
\end{proof}


\begin{remark}
When $\Delta>0$, Assumption \ref{assump} does not imply the
condition $b)$ in $i)$. In fact, in the case when $M=\mathbb{S}^2$,
we have $r_*=\pi$ and $\Delta=1$. Then let $\alpha=0.51$ and
$\rho=0.99\pi(1-(2\alpha)^{-1})$, then
$2\alpha\rho/(2\alpha-1)\in(\pi/2+1.5393,\pi-0.0314)$, but
$F_{\alpha,\rho,1}(\pi/2-\rho)\thickapprox0.2907>0$.
\end{remark}


\begin{remark}
It is easily seen that if we replace $r_*$ by any $r\in(0,r_*]$ in
Assumption \ref{assump}, then Lemma \ref{of sense} still holds when
$r_*$ is replaced by $r$. This observation can be used to reinforce
the conclusions of Theorem \ref{position}. For example, in the case
when $\Delta>0$,
$$\frac{2\alpha\rho}{2\alpha-1}\leq\frac{r_*}{2}\,\,\text{implies that}\,\, Q_{\mu}\subset\bar{B}\bigg(a,\frac{1}{\sqrt{\Delta}}\arcsin\big(\frac{\alpha\sin(\sqrt{\Delta}\rho)}{\sqrt{2\alpha-1}}\big)\bigg)\subset
B\bigg(a,\frac{r_*}{4}\bigg).$$
\end{remark}

\begin{remark}
Although we have chosen the framework of this section to be a
Riemannian manifold, the essential tool that has been used is the
hinge version of the triangle comparison theorem. Consequently, all
the results in this section remain true if $M$ is a CAT$(\Delta)$
space (see \cite[Chapter 2]{Bridson}) and $r_*$ is replaced by
$\pi/\sqrt{\Delta}$ in Assumption \ref{assump}.
\end{remark}

\begin{remark}
For the case when $\alpha=1$, Assumption \ref{assump} becomes
$$\rho<\frac{1}{2}\min\{\frac{\pi}{\sqrt{\Delta}}\,\,,\inj\,\}.$$
Observe that in this case, when $\Delta>0$, the condition
$F_{1,\rho,\Delta}(r_*/2-\rho)\leq0$ is trivially true in case of
need. Hence Theorem \ref{position} yields that
$Q_{\mu}\subset\bar{B}(a,\rho)$, which is exactly what the Theorem
2.1 in \cite{Afsari} says for medians.
\end{remark}

\section{uniqueness of fr\'{e}chet sample medians in compact\\ riemannian manifolds }

In this section, we shall always assume that $M$ is a complete
Riemannian manifold of dimension $l\geq2$. The Riemannian metric and
the Riemannian distance are denoted by
$\langle\,\cdot\,,\cdot\,\rangle$ and $d$, respectively. For each
point $x\in M$, $S_x$ denotes the unit sphere in $T_xM$. Moreover,
for a tangent vector $v\in S_x$, the distance between $x$ and its
cut point along the geodesic starting from $x$ with velocity $v$ is
denoted
by $\tau(v)$. Certainly, if there is no cut point along this geodesic, then we define $\tau(v)=+\infty$.

For every point $(x_1,\dots,x_N)\in M^N$, where $N\geq3$ is a fixed
natural number, we write
$$\mu(x_1,\dots,x_N)=\frac{1}{N}\sum_{k=1}^{N}\delta_{x_k}.$$
The set of all the Fr\'{e}chet medians of $\mu(x_1,\dots,x_N)$ is
denoted
by $Q(x_1,\dots,x_N)$.\\

We begin with the basic observation that if one data point is moved
towards a median along some minimizing geodesic for a little
distance, then the median remains unchanged.

\begin{proposition}\label{move towards median}
Let $(x_1,\dots,x_N)\in M^N$ and $m\in Q(x_1,\dots,x_N)$. Fix a
normal geodesic $\gamma:[0,+\infty)\rightarrow M$ such that
$\gamma(0)=x_1$, $\gamma(d(x_1,m))=m$. Then for every
$t\in[0,d(x_1,m)]$ we have 
\[
Q(\gamma(t),x_2,\dots,x_N)=
\begin{cases}
Q(x_1,\dots,x_N)\cap\gamma[t,\tau(\dot{\gamma}(0))],& \text{if $\tau(\dot{\gamma}(0))<+\infty$;}\\
Q(x_1,\dots,x_N)\cap\gamma_1[t,+\infty),& \text{if
$\tau(\dot{\gamma}(0))=+\infty$.}
\end{cases}
\]
Particularly, $m\in Q(\gamma(t),x_2,\dots,x_N)$.
\end{proposition}

\begin{proof}
For simplicity, let $\mu=\mu(x_1,\dots,x_N)$ and
$\mu_t=\mu(\gamma(t),x_2,\dots,x_N)$. Then for every $x\in M$,
\begin{align*}
\,f_{\mu_t}(x)-f_{\mu_t}(m)
=\,&\bigg(f_{\mu}(x)-\frac{1}{N}d(x,x_1)+\frac{1}{N}d(x,\gamma(t))\bigg)\\
&-\bigg(f_{\mu}(m)-\frac{1}{N}d(m,x_1)+\frac{1}{N}d(m,\gamma(t))\bigg)\\
=\,&\bigg(f_{\mu}(x)-f_{\mu}(m)\bigg)+\bigg(d(x,\gamma(t))+t-d(x,x_1)\bigg)\geq0.
\end{align*}
So that $m\in Q_{\mu_t}$. Combine this with the fact that $m$ is a
median of $\mu$, it is easily seen from the above proof that
$$Q_{\mu_t}= Q_{\mu}\cap\{x\in
M: d(x,\gamma(t))+t=d(x,x_1)\}.$$ Now the conclusion follows from
the definition of $\tau(\dot{\gamma}(0))$.
\end{proof}

The following theorem states that in order to get the uniqueness of
Fr\'{e}chet medians, it suffices to move two data points towards a
common median along some minimizing geodesics for a little distance.

\begin{theorem}
Let $(x_1,\dots,x_N)\in M^N$ and $m\in Q(x_1,\dots,x_N)$. Fix two
normal geodesics $\gamma_1,\gamma_2:[0,+\infty)\rightarrow M$ such
that $\gamma_1(0)=x_1$, $\gamma_1(d(x_1,m))=m$, $\gamma_2(0)=x_2$
and $\gamma_2(d(x_2,m))=m$. Assume that
\[
x_2\notin
\begin{cases}
\gamma_1[0,\tau(\dot{\gamma}_1(0))],& \text{if $\tau(\dot{\gamma}_1(0))<+\infty$;}\\
\gamma_1[0,+\infty),& \text{if $\tau(\dot{\gamma}_1(0))=+\infty$.}
\end{cases}
\]
Then for every $t\in(0,d(x_1,m)]$ and $s\in(0,d(x_2,m)]$ we have
$$Q(\gamma_1(t),\gamma_2(s),x_3,\dots,x_N)=\{m\}.$$
\end{theorem}

\begin{proof}
Without loss of generality, we may assume that both
$\tau(\dot{\gamma}_1(0))$ and $\tau(\dot{\gamma}_2(0))$ are finite.
Applying Proposition \ref{move towards median} two times we get
$$Q(\gamma_1(t),\gamma_2(s),x_3,\dots,x_N)\subset Q(x_1,\dots,x_N)
\cap\gamma_1[t,\tau(\dot{\gamma}_1(0))]\cap\gamma_2[s,\tau(\dot{\gamma}_2(0))].$$
Since $x_2\notin \gamma_1[0,\tau(\dot{\gamma}_1(0))]$, the
definition of cut point yields
$\gamma_1[t,\tau(\dot{\gamma}_1(0))]\cap\gamma_2[s,\tau(\dot{\gamma}_2(0))]=\{m\}$.
The proof is complete.
\end{proof}

We need the following necessary conditions of Fr\'{e}chet medians.

\begin{proposition}\label{two cases}
Let $(x_1,\dots,x_N)\in M^N$ and $m\in Q(x_1,\dots,x_N)$. For every
$k=1,\dots,N$ let $\gamma_k:[0,d(m,x_k)]\rightarrow M$ be a
normal geodesic such that $\gamma_k(0)=m$ and $\gamma_k(d(m,x_k))=x_k$.\\
i) If $m$ does not coincide with any $x_k$, then
\begin{equation}\label{no data}
\sum_{k=1}^{N}\dot{\gamma}_k(0)=0.
\end{equation}
In this case, the minimizing geodesics $\gamma_1,\dots,\gamma_N$
are uniquely determined.\\
ii) If $m$ coincides with some $x_{k_0}$, then
\begin{equation}\label{data}
\bigg|\sum_{x_k\neq
x_{k_0}}\dot{\gamma}_k(0)\bigg|\leq\sum_{x_k=x_{k_0}}1.
\end{equation}
\end{proposition}

\begin{proof}
For sufficiently small $\varepsilon>0$, Proposition \ref{move
towards median} yields that $m$ is a median of
$\mu(\gamma_1(\varepsilon),\dots,\gamma_N(\varepsilon))$. Hence
\cite[Theorem 2.2]{Yang} gives \eqref{no data} and \eqref{data}. Now
assume that $m$ does not coincide with any $x_k$ and, without loss
of generality, there is another normal geodesic
$\zeta_1:[0,d(m,x_1)]\rightarrow M$ such that $\zeta_1(0)=m$ and
$\zeta_1(d(m,x_1))=x_1$. Then \eqref{no data} yields
$\dot{\zeta}_1(0)+\sum_{k=2}^{N}\dot{\gamma}_k(0)=0.$ So that
$\dot{\zeta}_1(0)=\dot{\gamma}_1(0)$, that is to say,
$\zeta_1=\gamma_1$. The proof is complete.
\end{proof}

From now on, we will only consider the case when $M$ is a compact
Riemannian manifold. As a result, let the following assumption hold
in the rest part of this section:
\begin{Assumption}
$M$ is a compact Riemannian manifold with diameter $L$.
\end{Assumption}
In what follows, all the measure-theoretic statements should be
understood to be with respect to the canonical Lebesgue measure of
the underlying manifold. Let $\lambda_{M}$ denote the canonical
Lebesgue measure of $M$.

The following lemma gives a simple observation on dimension.

\begin{lemma}\label{dim}
The manifold
\begin{align*}
V=\bigg\{(x,n_1,\dots,n_N):\,\, &x\in M, \,(n_k)_{1\leq k\leq
N}\subset S_x,\,\, \sum_{k=1}^{N}n_k=0\,\, \text{and there exist}\\
&k_1\neq k_2 \,\,\text{such that}\,\,
n_{k_1}\,\,\text{and}\,\,n_{k_2}\,\,\text{are linearly
independent}\bigg\}
\end{align*}
is of dimension $N(l-1)$.
\end{lemma}


\begin{proof}
Consider the manifold \begin{align*} W=\bigg\{(x,n_1,\dots,n_N):\,\,
&x\in M, \,(n_k)_{1\leq k\leq N}\subset S_x\,\,\text{and there
exist}
\,\,k_1\neq k_2\\
&\text{such that}\,\, n_{k_1}\,\,\text{and}\,\,n_{k_2}\,\,\text{are
linearly independent}\bigg\},
\end{align*}
then the smooth map
$$f:\quad W\longrightarrow TM,\quad
(x,n_1,\dots,n_N)\longmapsto\sum_{k=1}^{N}n_k$$ has full rank $l$
everywhere on $W$. Now the constant rank level set theorem (see
\cite[Theorem 8.8]{Lee}) yields the desired result.
\end{proof}

Our method of studying the uniqueness of Fr\'{e}chet medians is
based on the regularity properties of the following function:
\begin{align*}
\varphi:\quad &V\times\mathbf{R}^N\longrightarrow M^N,\\
&(x,n_1,\dots,n_N,r_1,\dots,r_N)\longmapsto(\exp_x(r_1n_1),\dots,\exp_x(r_Nn_N)).
\end{align*}
For a closed interval $[a,b]\subset\mathbf{R}$, the restriction of
$\varphi$ to $V\times[a,b]^N$ will be denoted by $\varphi_{a,b}$.
The canonical projection of $V\times\mathbf{R}^N$ onto $M$ and
$\mathbf{R}^N$ will be denoted by $\sigma$ and $\zeta$,
respectively.

Generally speaking, the non uniqueness of Fr\'{e}chet medians is due
to some symmetric properties of data points. As a result, generic
data points should have a unique Fr\'{e}chet median. In mathematical
language, this means that the set of all the particular positions of
data points is of measure zero. Now our aim is to find all these
particular cases. Firstly, in view of the uniqueness result of
Riemannian medians (see \cite[Theorem 3.1]{Yang}), the first null
set that should be eliminated is
$$C_1=\bigg\{(x_1,\dots,x_N)\in M^N:
x_1,\dots,x_N \,\,\text{are contained in a single
geodesic}\bigg\}.$$ Observe that $C_1$ is a closed subset of $M^N$.
The second null set coming into our sight is the following one:
$$C_2=\bigg\{(x_1,\dots,x_N)\in M^N:\,\,(x_1,\dots,x_N)\,\,\text{is a critical value of}\,\,\varphi_{0,L}\bigg\}.$$
Since $\varphi$ is smooth, Sard's theorem implies that $C_2$ is of
measure zero. Moreover, it is easily seen that $(M^N\setminus
C_1)\cap C_2$ is closed in $M^N\setminus C_1$.

The following proposition says that apart form $C_1\cup C_2$ one
can only have a finite number of Fr\'{e}chet medians. 

\begin{proposition}\label{finite}
$Q(x_1,\dots,x_N)$ is a finite set for every $(x_1,\dots,x_N)\in
M^N\setminus (C_1\cup C_2)$.
\end{proposition}

\begin{proof}
Let $(x_1,\dots,x_N)\in M^N\setminus (C_1\cup C_2)$ and $A\subset
Q(x_1,\dots,x_N)$ be the set of medians that do not coincide with
any $x_k$. If $A=\phi$, then there is nothing to prove. Now assume
that $A\neq\phi$, then Proposition \ref{two cases} implies that
$A\subset\sigma\circ\varphi_{0,L}^{-1}(x_1,\dots,x_N)$. Moreover,
Lemma \ref{dim} and the constant rank level set theorem imply that
$\varphi_{0,L}^{-1}(x_1,\dots,x_N)$ is a zero dimensional regular
submanifold of $V\times[0,L]^N$, that is, some isolated points.
Since $(x_1,\dots,x_N)\notin C_1$,
$\varphi_{0,L}^{-1}(x_1,\dots,x_N)$ is also compact, hence it is a
finite set. So that $A$ is also finite, as desired.
\end{proof}

The following two lemmas enable us to avoid the problem of cut
locus.
\begin{lemma}\label{diffeo null}
Let $U$ be a bounded open subset of $V\times\mathbf{R}^N$ such that
$\varphi:\,\, U\longmapsto \varphi(U)$ is a diffeomorphism, then
$$\lambda_{M}^{\otimes N}\bigg\{(x_1,\dots,x_N)\in
\varphi(U):\sigma\circ\varphi^{-1}(x_1,\dots,x_N)\in\bigcup_{k=1}^{N}\cut(x_k)\bigg\}=0.$$

\end{lemma}

\begin{proof}
Without loss of generality, we will show that $\lambda_{M}^{\otimes
N}\{(x_1,\dots,x_N)\in
M^N:\sigma\circ\varphi^{-1}(x_1,\dots,x_N)\in\cut(x_N)\}=0$. In
fact, letting $(x_1,\dots,x_N)=\varphi(x,n_1,\dots
n_N,r_1,\dots,r_N)$, $\textbf{n}=(n_1,\dots n_N)$ and $\det
(D\varphi)\leq c$ on $U$ for some $c>0$, then the change of variable
formula and Fubini's theorem yield that
\begin{align*}
&\lambda_{M}^{\otimes N}\{(x_1,\dots,x_N)\in\varphi(U)
:\sigma\circ\varphi^{-1}(x_1,\dots,x_N)\in\cut(x_N)\}\\
=\,&\lambda_{M}^{\otimes N}\{(x_1,\dots,x_N)\in
\varphi(U):x_N\in\cut(\sigma\circ\varphi^{-1}(x_1,\dots,x_N))\}\\
=\,&\int_{\varphi(U)}\textbf{1}_{\{x_N\in\cut(\sigma\circ\varphi^{-1}(x_1,\dots,x_N))\}}dx_1\dots
dx_N\\
=\,&\int_U
\textbf{1}_{\{\exp_x(r_Nn_N)\in\cut(x)\}}\det(D\varphi)dx\,d\textbf{n}\,dr_1\dots
dr_N\\
\leq\,&c\int_{V\times \mathbf{R}^N}
\textbf{1}_{\{\exp_x(r_Nn_N)\in\cut(x)\}}dx\,d\textbf{n}\,dr_1\dots
dr_N\\
=\,&c\int_{V\times \mathbf{R}^{N-1}}dx\,d\textbf{n}\,dr_1\dots
dr_{N-1}\int_{\mathbf{R}} \textbf{1}_{\{\exp_x(r_Nn_N)\in\cut(x)\}}
dr_N\\
=\,& 0.
\end{align*}
The proof is complete.
\end{proof}

In order to tackle the cut locus, it is easily seen that the
following null set should be eliminated:
$$C_3=\bigg\{(x_1,\dots,x_N)\in
M^N:\,\,x_i\in\{x_j\}\cup\cut(x_j),\,\,\text{for some}\,\,i\neq
j\bigg\}.$$ Observe that $C_3$ is also closed because the set
$\{(x,y)\in M^2:\,\,x\in\cut(y)\}$ is closed.

\begin{lemma}\label{diffeo null 1}
For every $(x_1,\dots,x_N)\in M^N\setminus(C_1\cup C_2\cup C_3)$,
there exists $\delta>0$ such that
$$\lambda_{M}^{\otimes N}\bigg\{(y_1,\dots,y_N)\in
B(x_1,\delta)\times\dots\times
B(x_N,\delta):\,\,Q(y_1,\dots,y_N)\cap\bigcup_{k=1}^{N}\cut(y_k)\neq\phi\bigg\}=0.$$
\end{lemma}

\begin{proof}
If $Q(x_1,\dots,x_N)\subset\{x_1,\dots,x_N\}$, then the assertion is
trivial by Theorem \ref{consistency}. Now assume that
$Q(x_1,\dots,x_N)\setminus\{x_1,\dots,x_N\}\neq\phi$, then the proof
of Proposition \ref{finite} yields that
$\varphi_{0,L}^{-1}(x_1,\dots,x_N)$ is finite. Hence we can choose
$\varepsilon,\eta>0$ and $O$ a relatively compact open subset of $V$
such that $\varepsilon<\min\{\,\inj
M/2,\min\{r_k:(r_1,\dots,r_N)\in\zeta\circ\varphi_{0,L}^{-1}(x_1,\dots,x_N),k=1,\dots,N
\}$\}, $B(x_i,2\varepsilon)\cap \cut(B(x_j,2\varepsilon))=\phi$ and
$\varphi_{\varepsilon,L+\eta}^{-1}$ $(x_1,\dots,x_N)\subset
O\times(\varepsilon,L+\eta)^N$. Then by the stack of records theorem
(see \cite[Exercise 7, Chapter 1, Section 4]{Guillemin}), there
exists $\delta\in(0,\varepsilon)$ such that
$U=B(x_1,\delta)\times\dots\times B(x_N,\delta)$ verifies
$\varphi_{\varepsilon,L+\eta}^{-1}(U)=V_1\cup\dots\cup V_h$,
$V_i\cap V_j=\phi$ for $i\neq j$ and $\varphi_{\varepsilon,L+\eta}:
V_i\rightarrow U$ is a diffeomorphism for every $i$. Lemma
\ref{diffeo null} yields that there exists a null set $A\subset U$,
such that for every $(y_1,\dots,y_N)\in U\setminus A$ and for every
$(y,n_1,\dots,n_N,r_1,\dots,r_N)\in\varphi_{\varepsilon,L+\eta}^{-1}(y_1,\dots,y_N)$
one always has $y\notin\bigcup_{k=1}^{N}\cut(y_k)$. Particularly,
for $m\in Q(y_1,\dots,y_N)$ such that $d(m,y_k)\geq\varepsilon$ for
every $k$, we have $m\notin\bigcup_{k=1}^{N}\cut(y_k)$. Now let
$m\in Q(y_1,\dots,y_N)$ such that $d(m,y_{k_0})<\varepsilon$ for
some $k_0$, then $d(m,x_{k_0})\leq
d(m,y_{k_0})+d(y_{k_0},x_{k_0})<2\varepsilon$. So that
$m\notin\bigcup_{k=1}^{N}\cut(y_k)$ since $y_k\in
B(x_k,2\varepsilon)$. This completes the proof.
\end{proof}

Now the cut locus can be eliminated without difficulty.

\begin{proposition}\label{cut locus}
The set  $$C_4=\bigg\{(x_1,\dots,x_N)\in
M^N:\,\,Q(x_1,\dots,x_N)\cap\bigcup_{k=1}^{N}\cut(x_k)\neq\phi\bigg\}$$
is of measure zero and is closed.
\end{proposition}
\begin{proof}
It suffices to show that $C_4$ is of measure zero. This is a direct
consequence of Lemma \ref{diffeo null 1} and the fact that
$M\setminus(C_1\cup C_2\cup C_3)$ is second countable.
\end{proof}

Let $x,y\in M$ such that $y\notin\{x\}\cup\cut(x)$, we denote
$\gamma_{xy}:[0,d(x,y)]\rightarrow M$ the unique minimizing geodesic
such that $\gamma_{xy}(0)=x$ and $\gamma_{xy}(d(x,y))=y$. For every
$u\in T_xM$ and $v\in T_yM$, let $J(v,u)(\cdot)$ be the unique
Jacobi field along $\gamma_{xy}$ with boundary condition
$J(u,v)(0)=u$ and $J(u,v)(d(x,y))=v$.

\begin{lemma}\label{derivative}
Let $x,y\in M$ such that $y\notin\{x\}\cup\cut(x)$. Then for every
$v\in T_yM$, we have
$$\nabla_v\frac{\exp_x^{-1}(\cdot)}{d(x,\cdot)}=\dot{J}(0_x,v^{\nor})(0),$$
where $v^{\nor}$ is the normal component of $v$ with respect to
$\dot{\gamma}_{xy}(d(x,y))$.
\end{lemma}

\begin{proof}
By \cite[p. 1517]{Arnaudon3},
\begin{align*}
\nabla_v\frac{\exp_x^{-1}(\cdot)}{d(x,\cdot)}=&\frac{\nabla_v
\exp_x^{-1}(\cdot)}{d(x,y)}-\frac{\exp_x^{-1}y\nabla_v
d(x,\cdot)}{d(x,y)^2}\\
=&\frac{\nabla_v
\exp_x^{-1}(\cdot)}{d(x,y)}-\bigg\langle v,\frac{-\exp_y^{-1}x}{d(x,y)}\bigg\rangle\frac{\exp_x^{-1}y}{d(x,y)^2}\\
=&\dot{J}(0_x,v)(0)-\dot{J}(0_x,v^{\tan})(0)\\
=&\dot{J}(0_x,v^{\nor})(0),
\end{align*}
where $v^{\tan}$ is the tangent component of $u$ with respect to
$\dot{\gamma}_{xy}(d(x,y))$.
\end{proof}

With this differential formula, another particular case can be
eliminated now.
\begin{proposition}\label{sum norm 1}
The set
$$C_5=\bigg\{(x_1,\dots x_N)\in M^N\setminus (C_1\cup C_3): \bigg|\sum_{k\neq
k_0}\frac{\exp_{x_{k_0}}^{-1}x_k}{d(x_{k_0},x_k)}\bigg|=1\,\,\text{for
some}\,\,k_0\bigg\}$$ is of measure zero and is closed in
$M^N\setminus (C_1\cup C_3)$.
\end{proposition}

\begin{proof}
Without loss of generality, let us show that
$$C^{\prime}_5=\bigg\{(x_1,\dots,x_N)\in M^N\setminus
(C_1\cup C_3):h(x_1,\dots,x_N)=1\bigg\}$$ is of measure zero, where
$$h(x_1,\dots,x_N)=\bigg|\sum_{k=1
}^{N-1}\frac{\exp_{x_{N}}^{-1}x_k}{d(x_{N},x_k)}\bigg|^2.$$ By the
constant rank level set theorem,  it suffices to show that $\grad h$
is nowhere vanishing on $M^N\setminus (C_1\cup C_3)$. To this end,
let $(x_1,\dots,x_N)\in M^N\setminus (C_1\cup C_3)$ and $u=\sum_{k=1
}^{N-1}\exp_{x_{N}}^{-1}x_k/d(x_{N},x_k)$. Since $N\geq 3$, without
loss of generality, we can assume that $u$ and
$\exp_{x_{N}}^{-1}x_1$ are not parallel. Then for each $v\in
T_{x_1}M$, by lemma \ref{derivative} we have
\begin{align*}
\nabla_v h(\cdot,x_2,\dots,x_N)=
\bigg\langle\nabla_v\frac{\exp_{x_{N}}^{-1}x_1}{d(x_{N},x_1)},u\bigg\rangle
=2\langle
\dot{J}(0_{x_N},v^{\nor})(0),u\rangle=2\langle\psi(v),u\rangle,
\end{align*}
where the linear map $\psi$ is defined by
$$\psi:\quad
T_{x_1}M\longrightarrow T_{x_N}M,\quad
v\longmapsto\dot{J}(0_{x_N},v^{\nor})(0),$$ $v^{\nor}$ is the normal
component of $v$ with respect to $\exp_{x_1}^{-1}x_{N}$. Hence we
have $\grad_{x_1}h(\cdot,x_2,\dots,x_N)=\psi^*(u)$, where $\psi^*$
is the adjoint of $\psi$. Since the range space of $\psi$ is the
orthogonal complement of $\exp_{x_{N}}^{-1}x_1$, one has necessarily
$\psi^*(u)\neq0$, this completes the proof.
\end{proof}

The reason why the set $C_5$ should be eliminated is given by the
following simple lemma.

\begin{lemma}\label{norm 1}
Let $(x_1,\dots,x_N),(x_1^i,\dots,x_N^i)\in M^N\setminus C_3$ for
every $i\in\mathbf{N}$ and $(x_1^i,\dots,x_N^i)\longrightarrow
(x_1,\dots,x_N)$, when $i\longrightarrow\infty$. Assume that $m_i\in
Q(x_1^i,\dots,x_N^i)\setminus\{x_1^i,\dots,x_N^i\}$ and
$m_i\longrightarrow x_{k_0}$, then $$\bigg|\sum_{k\neq k_0
}\frac{\exp_{x_{k_0}}^{-1}x_k}{d(x_{k_0},x_k)}\bigg|=1.$$
\end{lemma}

\begin{proof}
It suffices to note that for $i$ sufficiently large, Proposition
\ref{two cases} gives
$$\bigg|\sum_{k\neq k_0
}\frac{\exp_{m_i}^{-1}x^i_k}{d(m_i,x^i_k)}\bigg|=\bigg|\frac{\exp_{m_i}^{-1}x_{k_0}^i}{d(m_i,x_{k_0}^i)}\bigg|=1.$$
Then letting $i\rightarrow\infty$ gives the result.
\end{proof}

As a corollary to Proposition \ref{sum norm 1}, the following
proposition tells us that for generic data points, there cannot
exist two data points which are both Fr\'{e}chet medians.

\begin{proposition}
The set
\begin{align*}
C_6=\bigg\{&(x_1,\dots,x_N)\in M^N\setminus(C_1\cup C_3\cup C_5):\\
&\text{there exist}\,\,i\neq j\,\,\text{such that}\,\,
f_{\mu}(x_i)=f_{\mu}(x_j),\,\,\text{where}\,\,\mu=\mu(x_1\dots,x_N)\bigg\}
\end{align*}
is of measure zero and is closed in $M^N\setminus(C_1\cup C_3\cup
C_5)$.
\end{proposition}

\begin{proof}
For every $(x_1,\dots,x_N)\in M^N\setminus(C_1\cup C_3\cup C_5)$,
let $f(x_1,\dots,x_N)=f_{\mu}(x_{N-1})$ and
$g(x_1,\dots,x_N)=f_{\mu}(x_{N})$, where $\mu=\mu(x_1\dots,x_N)$.
Without loss of generality, we will show that $\{(x_1,\dots,x_N)\in
M^N\setminus(C_1\cup C_3\cup
C_5):f(x_1,\dots,x_N)=g(x_1,\dots,x_N)\}$ is of measure zero. Always
by the constant rank level set theorem, it suffices to show that
$\grad f$ and $\grad g$ are nowhere identical on
$M^N\setminus(C_1\cup C_3\cup C_5)$. In fact,
\begin{align*}
\grad_{x_N}
f(x_1,\dots,x_N)&=\frac{-\exp_{x_N}^{-1}x_{N-1}}{d(x_N,x_{N-1})}\\
&\neq\sum_{k=1}^{N-1}\frac{-\exp_{x_N}^{-1}x_{k}}{d(x_N,x_{k})}
=\grad_{x_N} g(x_1,\dots,x_N),
\end{align*}
because $C_5$ is eliminated, as desired.
\end{proof}

As needed in the following proofs, the restriction of $\varphi$ on
the set
$$E=\bigg\{(x,n_1,\dots,n_N,r_1,\dots,r_N)\in V\times\mathbf{R}^N:
0<r_k<\tau(n_k)\,\,\text{for every}\,\,k\bigg\}$$ is denoted by let
$\hat{\varphi}$. Clearly, $\hat{\varphi}$ is smooth.

The lemma below is a final preparation for the main result of this
section.

\begin{lemma}\label{not equal}
Let $U$ be an open subset of $M^N\setminus\bigcup_{k=1}^{6}C_k$.
Assume that $U_1\cup U_2\subset\hat{\varphi}^{-1}(U)$ such that for
$i=1,2$, $\hat{\varphi}_i=\hat{\varphi}|_{U_i}:U_i\rightarrow U$ is
a diffeomorphism and $\sigma(U_1)\cap\sigma(U_2)=\phi$. For
simplicity, when $(x_1,\dots,x_N)\in U$, we write
$x=\sigma\circ\hat{\varphi}_1^{-1}(x_1,\dots,x_N)$,
$y=\sigma\circ\hat{\varphi}_2^{-1}(x_1,\dots,x_N)$ and
$\mu=\mu(x_1,\dots,x_N)$. Then the following two sets are of measure
zero:
\begin{align*}
&\bigg\{(x_1,\dots,x_N)\in
U:f_{\mu}(x)=f_{\mu}(y)\bigg\}\,\,\text{and}\\
&\bigg\{(x_1,\dots,x_N)\in U:\,\,\text{there
exists}\,\,k_0\,\,\text{such
that}\,\,f_{\mu}(x)=f_{\mu}(x_{k_0})\bigg\}.
\end{align*}
\end{lemma}

\begin{proof}
We only show the first set is null, since the proof for the second
one is similar. Let $f_1(x_1,\dots,x_N)=f_{\mu}(x)$,
$f_2(x_1,\dots,x_N)=f_{\mu}(y)$ and $w_k\in T_{x_k}M$. Then the
first variational formula of arc length (see \cite[p. 5]{Cheeger})
yields that
\begin{align*}
&\frac{d}{dt}\bigg|_{t=0}f_1(\exp_{x(t)}(tw_1),\dots\exp_{x(t)}(tw_N))\\
=&\sum_{k=1}^{N}
\bigg(\bigg\langle\frac{-\exp_{x_k}^{-1}x}{d(x_k,x)},w_k\bigg\rangle-\bigg\langle\dot{x}(0),\frac{\exp_{x}^{-1}x_k}{d(x,x_k)}\bigg\rangle\bigg)\\
=&\sum_{k=1}^{N}
\bigg\langle\frac{-\exp_{x_k}^{-1}x}{d(x_k,x)},w_k\bigg\rangle-\bigg\langle\dot{x}(0),\sum_{k=1}^{N}\frac{\exp_{x}^{-1}x_k}{d(x,x_k)}\bigg\rangle\\
=&\sum_{k=1}^{N}
\bigg\langle\frac{-\exp_{x_k}^{-1}x}{d(x_k,x)},w_k\bigg\rangle.
\end{align*}
Hence $$\grad
f_1(x_1,\dots,x_N)=\bigg(\frac{-\exp_{x_1}^{-1}x}{d(x_1,x)},\dots,\frac{-\exp_{x_N}^{-1}x}{d(x_N,x)}\bigg).$$
Observe that $(x_1,\dots,x_N)\notin C_1$, $N\geq3$ and $x\neq y$, we
have $\grad f_1\neq\grad f_2 $ on $U$. Then the constant rank level
set theorem yields that $\{f_1=f_2\}$ is a regular submanifold of
$U$ of codimension $1$, hence it is of measure zero. The proof is
complete.
\end{proof}

The following theorem is the main result of this section.

\begin{theorem}\label{unique}
$\mu(x_1,\dots,x_N)$ has a unique Fr\'{e}chet median for almost
every $(x_1,\dots,x_N)\in M^N$.
\end{theorem}

\begin{proof}
Since $M^N\setminus\bigcup_{k=1}^{6}C_k$ is second countable, it
suffices to show that for every $(x_1,\dots,x_N)\in
M^N\setminus\bigcup_{k=1}^{6}C_k$, there exists $\delta>0$ such that
$\mu(y_1,\dots,y_N)$ has a unique Fr\'{e}chet median for almost
every $(y_1,\dots,y_N)\in B(x_1,\delta)\times\dots\times
B(x_N,\delta)$ . In fact, let $(x_1,\dots,x_N)\in
M^N\setminus\bigcup_{k=1}^{6}C_k$, without loss of generality, we
can assume that $Q(x_1,\dots,x_N)=\{y,z,x_N\}$, where
$y,z\notin\{x_1,\dots,x_N\}$. Assume that
$Y=(y,n_1,\dots,n_N,r_1,\dots,r_N) $ and
$Z=(z,v_1,\dots,v_N,t_1,\dots,t_N)\in
\hat{\varphi}^{-1}(x_1,\dots,x_N)$. Since $(x_1,\dots,x_N)$ is a
regular value of $\hat{\varphi}$, we can choose a $\delta>0$ such
that there exist neighborhoods $U_1$ of $Y$ and $U_2$ of $Z$ such
that for $i=1,2$,
$\hat{\varphi}_i=\hat{\varphi}|_{U_i}:U_i\rightarrow U$ is
diffeomorphism, $\sigma(U_1)\cap\sigma(U_2)=\phi$ and
$B(x_N,\delta)\cap(\sigma(U_1)\cup\sigma(U_2))=\phi$, where
$U=B(x_1,\delta)\times\dots\times B(x_N,\delta)$. Furthermore, by
Theorem \ref{consistency} and Lemma \ref{norm 1}, we can also assume
that for every $(y_1,\dots,y_N)\in B(x_1,\delta)\times\dots\times
B(x_N,\delta)$, $Q(y_1,\dots,y_N)\subset
B(x_N,\delta)\cup\sigma(U_1)\cup\sigma(U_2)$ and
$Q(y_1,\dots,y_N)\cap B(x_N,\delta)\subset\{y_N\}$. Now it suffices
to use Lemma \ref{not equal} to complete the proof.
\end{proof}

\begin{remark}
In probability language, Theorem \ref{unique} is equivalent to say
that if $(X_1,\dots,X_N)$ is an $M^N$-valued random variable with
density, then $\mu(X_1,\dots,X_N)$ has a unique Fr\'{e}chet median
almost surely. Clearly, the same statement is also true if
$X_1,\dots,X_N$ are independent and $M$-valued random variables with
desity.
\end{remark}

\textsl{Acknowledgements} The author is very grateful to his PhD
advisors: Marc Arnaudon and Fr\'{e}d\'{e}ric Barbaresco for their
inspiring guidance, constant help and encouragement.

\end{document}